\documentclass{amsart}
\usepackage{amsmath, amsthm}
\usepackage{amssymb}
\usepackage{amsfonts}
\usepackage{latexsym}
\usepackage{mathrsfs}
\usepackage{bm}
\usepackage{graphicx}

\theoremstyle{plain}
\newtheorem{theorem}{Theorem}[section]

\newtheorem{lemma}[theorem]{Lemma}
\newtheorem{definition}[theorem]{Definition}
\newtheorem{proposition}[theorem]{Proposition}

\newtheorem{question}[theorem]{Question}

\theoremstyle{remark}
\newtheorem{remark}{Remark}[section]
\newtheorem{example}{Example}

\newcommand{\Mod}{\mbox{\rm Mod}}

\begin{document}

\date{}
\title[Parabolic isometry]
{Parabolic isometries of visible CAT(0) spaces and metrics on moduli space}

\author{Yunhui Wu}

\address{Department of Mathematics, Rice University, 6100 Main St, Houston, TX 77005}

\email{yw22@rice.edu}

\begin{abstract}
We show that the translation length of any parabolic isometry on a complete semi-uniformly visible CAT(0) space is always zero. As a consequence, we will classify the isometries on visible CAT(0) spaces in terms of translation lengths. We will also show that the moduli space $\mathbb{M}(S_{g,n})$ of surface $S_{g,n}$ of $g$ genus with $n$ punctures admits no complete visible CAT(0) Riemannian metric if $3g+n\geq 5$, which answers the Brock-Farb-McMullen question in the visible case. 
\end{abstract}

\maketitle

\section{Introduction}

CAT(0) spaces are generalizations of nonpositively curved Riemannian manifolds. Nonsmooth spaces, which are also called singular spaces, may occur in CAT(0) spaces. A classical example of a singular CAT(0) space is a tree, which is an one-dimensional graph without loops. The vertices are the singular set. A CAT(0) space could be no proper, that is, not locally compact. The properness of a tree depends on whether it is locally finite. Actually a locally infinite tree is the simplest example for CAT(0) spaces which are not locally compact. The first part of this paper will focus on parabolic isometries of complete CAT(0) spaces which may be no proper.  

Let $M$ be a complete CAT(0) space. An isometry $\gamma$ of $M$ is a map $\gamma:M \to M$ which satisfies $dist(\gamma \cdot p, \gamma \cdot q)=dist(p,q)$, for all $p,q\in M$. The set of isometries on a metric space is a group. An isometry can be classified as elliptic, hyperbolic, or parabolic. An isometry is called \textsl{elliptic} if it has a fixed point in $M$. Any finite order element in the isometry group of a complete CAT(0) space is always elliptic (see lemma \ref{fpt} in section \ref{np}). An isometry $\gamma$ is called \textsl{hyperbolic} if there exists a geodesic line $c: (-\infty,+\infty)\rightarrow M$ such that $\gamma$ acts on $c(\mathbb{R})$ by a non-trivial translation. The fundamental group of a closed nonpositively curved Riemannian manifold consists of hyperbolic isometries except the unit. If an isometry is neither elliptic nor hyperbolic, then we call it to be \textsl{parabolic}. Parabolic isometries may occur in the fundamental group of an open nonpositively curved Riemannian manifold.

Let $\gamma$ be an isometry of a complete CAT(0) space $M$. We define the \textsl{translation length} $|\gamma|$ of $\gamma$  as 
\[|\gamma|:=\inf_{p\in M} dist(\gamma \cdot p,p).\] 
From the definition of translation length, $|\gamma|$ may be not achieved. If $|\gamma|$ is achieved in $M$, $|\gamma|=0$ corresponds to the elliptic case, and $|\gamma|>0$ corresponds to the hyperbolic case. If $|\gamma|$ can not be achieved in $M$, $\gamma$ is parabolic. One can see more details in \cite{BH,BGS}. Let us look at the following two examples. Let $\mathbb{H}^2$ be the upper half plane endowed with the hyperbolic metric and define $\gamma :\mathbb{H}^2\rightarrow \mathbb{H}^2$ to be $\gamma\cdot(x,y)=(x+1,y)$. It is easy to see that $|\gamma|=0$, which can not be achieved in $\mathbb{H}$. Thus, $\gamma$ is parabolic. Similarly consider $\mathbb{R}\times \mathbb{H}^2$ and define $\gamma :\mathbb{R}\times \mathbb{H}^2\rightarrow \mathbb{R}\times \mathbb{H}^2$ to be $\gamma\cdot(z,(x,y))=(z+1,(x+1,y))$. It is easy to see that $\gamma$ is parabolic and $|\gamma|=1$. So parabolic isometry with positive translation length may occur in CAT(0) spaces. 

A visible CAT(0) space, introduced by Eberlein and O'Neill in \cite{EO}, needs the space to be more curved. In some sense it means that for any two different points at ``infinity", they can be viewed from each other along the space, which can not happen in $\mathbb{R}^{n}$. We call a complete CAT(0) space $M$ is \textsl{visible} if for any $x\neq y \in M(\infty)$ there exists a geodesic line $c:(-\infty,+\infty)\to M$ such that $c(-\infty)=x$ and $c(+\infty)=y$, where $M(\infty)$ is the visual boundary of $M$ consisting of asymptotic geodesic rays (see \cite{BH} for details). In particular, a complete CAT(0) space, whose visual boundary is empty, is a visible space. When the visual boundary is not empty, classical examples for visible spaces include trees, which are singular, and complete simply connected Riemannian manifolds whose sectional curvatures are bounded above by negative numbers, which are smooth.  

For the example $\mathbb{R}\times \mathbb{H}^2$ above which contains parabolic isometries with positive translation lengths, there exists embedding totally geodesical Euclidean planes in it. In particular, $\mathbb{R}\times \mathbb{H}^2$ is not visible. It was shown in \cite{Wu12} that any parabolic isometry of a complete proper visible CAT(0) space has zero translation length, which was also served as a bridge to find a counterexample for the Eberlein conjecture which said that \textsl{a complete open manifold $M$ with sectional curvature $-1\leq K_{M}\leq 0$ and finite volume is visible if the universal covering space $\tilde{M}$ of $M$ contains no imbedded flat half planes}. And the proof in \cite{Wu12} requires that the space is proper. So it is natural to ask
\begin{question}\label{gq-1}
Does every parabolic isometry of a complete visible CAT(0) space have zero translation length?
\end{question}

In general, the answer to the question above is not true, one can see the example in remark \ref{tushar}. If we just consider complete visible CAT(0) spaces, which may be no proper, we still would like to know what kind of spaces such that parabolic isometries have zero translation length. A semi-uniformly visible CAT(0) space, roughly speaking, means a visible CAT(0) space which can not contain flat strips with arbitrary large widths, one can see definition \ref{dosuvs} for details. The motivation that we introduce semi-uniformly visible CAT(0) spaces is that the counterexample in remark \ref{tushar} for question \ref{gq-1} contains flat strips with arbitrary large widths. Proposition \ref{sosuv} in section 2 tells that both complete proper visible CAT(0) spaces and complete Gromov-hyperbolic CAT(0) spaces are semi-uniformly visible CAT(0) spaces. The following theorem gives an affirmative answer to question \ref{gq-1} for semi-uniformly visible CAT(0) spaces.  

\begin{theorem}\label{mt-1}
Let $M$ be a complete semi-uniformly visible CAT(0) space. Then any parabolic isometry has zero translation length. i.e., for any parabolic isometry $\gamma$ on $M$ we have $|\gamma|=0$.
\end{theorem}

Buyalo in \cite{BU} proved theorem \ref{mt-1} when $M$ is a complete Gromov-hyperbolic CAT(0) space. Since a complete Gromov-hyperbolic CAT(0) space is a complete semi-uniformly visible CAT(0) space (see proposition \ref{gmhiv} in section \ref{np}), theorem \ref{mt-1} gives a generalization. And the method here is different as what Buyalo did in \cite{BU}.  For the manifold case, Bishop and O'Neill in \cite{BO} proved that any parabolic isometry of a complete simply connected manifold M with sectional curvature $K_{M}\leq -1$ has zero translation length. In \cite{HH} Heintze and Hof used the geometry on horospheres to give a new proof for the manifold case.
\newline 


As a consequence of theorem \ref{mt-1}, we have the following classification on isometries of a complete semi-uniformly visible CAT(0) space in terms of translation lengths.

\begin{theorem}\label{class}
Let $M$ be a complete semi-uniformly visible CAT(0) space, and $\gamma$ be an isometry of $M$. Then,

(1): $\gamma$ is elliptic if and only if $\gamma$ has a fixed point in $M$.

(2): $\gamma$ is hyperbolic if and only if $|\gamma|>0$.

(3): $\gamma$ is parabolic if and only if $|\gamma|=0$ and $\gamma$ does not have any fixed point in $M$.
\end{theorem}

Let $S_{g,n}$ be a surface of $g$ genus with $n$ punctures, and $\mathbb{M}(S_{g,n})$ be the moduli space of $S_{g,n}$. There are a lot of canonical metrics on $\mathbb{M}(S_{g,n})$ like the Teichm\"uller metric, the Weil-Petersson metric, and so on (see \cite{IT92}). Kravetz in \cite{Kravetz} proved the Teichm\"uller metric has negative curvature. However, a mistake in the proof of Kravetz's theorem was found by Linch in \cite{Linch}. Masur in \cite{Masur75} proved that the Teichm\"uller metric is not nonpositively curved except several cases. Although the Weil-Petersson metric is negatively curved (see \cite{Wolpert86}), but not complete (see \cite{Wolpert75}). The McMullen metric, constructed by McMullen in \cite{McMullen00}, is a complete K\"ahler-hyperbolic metric in the sense of Gromov. Liu, Sun and Yau in \cite{LSY05} perturbed the Weil-Petersson metric to construct the so-called perturbed Ricci metric, which is complete and whose Ricci curvatures are pinched by two negative numbers. There is a question in Brock-Farb's paper \cite{BF} which states
\begin{question}[Brock-Farb-McMullen]
Does $\mathbb{M}(S_{g,n})$ admit a complete, nonpositively curved Riemannian metric?
\end{question}

Since $\mathbb{M}(S_{g,n})$ is an orbifold, here a Riemannian metric on $\mathbb{M}(S_{g,n})$ means a Riemannian metric on the Teichm\"uller space $\mathbb{T}(S_{g,n})$, the universal covering space of $\mathbb{M}(S_{g,n})$, on which the natural action of the mapping class group $\Mod_{S_{g,n}}$ on $\mathbb{T}(S_{g,n})$ is an isometric action. 

We call a Riemannian metric on $\mathbb{M}(S_{g,n})$ is called \textsl{visible CAT(0) Riemannian} if the sectional curvature of $\mathbb{T}(S_{g,n})$ is nonpositive and $\mathbb{T}(S_{g,n})$ is visible. If $3g+n\leq 4$ and $\mathbb{M}(S_{g,n})$ has positive dimension, then $(g,n)$ must be one of $\{(1,0),(1,1),(0,4)\}$. For these three cases, it is well known that the Teichm\"uller metric on $\mathbb{M}(S_{g,n})$ is a complete hyperbolic metric (see \cite{FW10}). In particular, the Teichm\"uller metric on $\mathbb{M}(S_{g,n})$ is complete visible CAT(0) Riemannian. Hence, we can always assume that $3g+n\geq 5$ for Brock-Farb-McMullen's question. The following theorem answers Brock-Farb-McMullen's question in the visible case. 

\begin{theorem}\label{mt-2}
If $3g+n\geq 5$, then $\mathbb{M}(S_{g,n})$ admits no complete visible CAT(0) Riemannian metric.  
\end{theorem}

Since a complete Riemannian manifold with sectional curvatures bounded above by a negative number is a visible CAT(0) manifold, the following theorem follows immediately from theorem \ref{mt-2}.
\begin{theorem}\label{ccnba}
If $3g+n\geq 5$, $\mathbb{M}(S_{g,n})$ admits no complete Riemannian metric such that the sectional curvature $K(\mathbb{M}(S_{g,n}))\leq -1$.
\end{theorem}

Ivanov in \cite{Iv88} showed that $\mathbb{M}(S_{g,n})$ $(3g+n\geq 5)$ admits no complete, finite volume Riemannian metric whose sectional curvature is pinched by two negative numbers. McMullen in \cite{McMullen00} stated that $\mathbb{M}(S_{g,n})$ $(3g+n\geq 5)$ admits no complete Riemannian metric whose sectional curvature is pinched by two negative numbers, which was proved by Brock and Farb in \cite{BF}. Moreover, the authors in \cite{BF} showed that  $\mathbb{M}(S_{g,n})$ $(3g+n\geq 5)$ admits no complete, finite volume Riemannian metric such that the universal covering space is Gromov-hyperbolic. For the visible and finite volume case, one can refer to \cite{Wu12}. What is new for theorem \ref{mt-2} and \ref{ccnba} is that there is no finite volume restriction. For related topics, one can also see \cite{Ji02, KM99, KN04, LSY04, McP99, MW95}.    
\newline

Throughout this paper, we always assume that the geodesics use arc-length parameters.
\newline

\textbf{Plan of the paper.} In section \ref{np} we set out necessary backgrounds, and prove some basic properties on CAT(0) spaces and mapping class groups, which will be applied in subsequent sections. Section \ref{mt12} establishes theorem \ref{mt-1} and theorem \ref{class}. Theorem \ref{mt-2} is proved in section \ref{somp}. 

\ \
\newline
\textbf{Acknowledgments.} The author is indebted to Andy Putman for the discussions on the proof of theorem \ref{mt-2}, in particular for his suggestion on writing this article and correction on English for the original manuscript. Thank also to Benson Farb and Mike Wolf for the discussions on the proof of theorem \ref{mt-2}.  The author also would like to thank Tushar Das for the discussions on theorem \ref{mt-1} and sharing his idea in remark \ref{tushar}.

\section{Notations and Preliminaries}\label{np}

\subsection{CAT(0) spaces} A CAT(0) space is a geodesic metric space in which each geodesic triangle is no fatter than a triangle in the Euclidean plane with the same edge lengths. 
\begin{definition} 
let $M$ be a geodesic metric space. For any $a,b,c \in M$, three geodesics $[a,b],[b,c],[c,a]$ form a geodesic triangle $\Delta.$ Let $\overline{\Delta}(\overline{a},\overline{b},\overline{c})\subset \mathbb{R}^2 $ be a triangle in the Euclidean plane with the same edge lengths as $\Delta$. Let $p,q$ be points on $[a,b]$ and $[a,c]$ respectively, and let $\overline{p},\overline{q}$ be points on $[\overline{a},\overline{b}]$ and $[\overline{a},\overline{c}]$ respectively, such that $dist_{M}(a,p)=dist_{\mathbb{R}^2}(\overline{a},\overline{p}), dist_{M}(a,q)=dist_{\mathbb{R}^2}(\overline{a},\overline{q})$. We call $M$ a \textbf{CAT(0) space} if for all $\Delta$ the inequality $dist_{M}(p,q)\leq dist_{\mathbb{R}^2}(\overline{p},\overline{q})$ holds.
\end{definition}

Let $M$ be a complete CAT(0) space. The ideal boundary, denoted by $M(\infty)$, consists of asymptotic rays. For each point $p \in M$ and $x \in M(\infty)$, since the distance function between geodesics is convex, it is not hard to see that there exists a unique geodesic ray $c$ which represents $x$ and starts from $p$ (see \cite{BH}). We write $c(+\infty)=x$. Although a complete CAT(0) space may be singular, the definition of CAT(0) spaces can guarantee that the notation of the angle, like the smooth case, still make sense (see \cite{BH}).  Given two points $x,y$ in $M(\infty)$. Let $\angle_{p}(x,y)$ denote the angle at $p$ between the unique geodesics rays which issue from $p$ and lie in the classes $x$ and $y$ respectively. The angular metric is defined to be $\angle(x,y):=\sup_{p\in M}{\angle_{p}(x,y)}$. It is easy to see that $\angle(x,y)=0$ if and only if $x=y$. On a complete visible CAT(0) space $M$, for any $x\neq y \in M(\infty)$, $\angle(x,y)=\pi$. So the angular metric gives a discrete topology on the ideal boundary of a complete visible CAT(0) space. 

The following lemma will be used in next section, which gives us a way to compute the angular metric.
\begin{lemma}\label{pro:angular}
Let $M$ be a complete CAT(0) space with a basepoint $p$. Let $x,y \in M(\infty)$ and $c,c'$ be two geodesic rays with $c(0)=c'(0)=p,$ $c(+\infty)=x$ and $c'(+\infty)=y$. Then,
\begin{eqnarray*}
2\sin(\frac{\angle(x,y)}{2})=\lim_{t\rightarrow +\infty}\frac{dist(c(t),c'(t))}{t}.
\end{eqnarray*}
\end{lemma}
\begin{proof}
See proposition 9.8 on page 281 in chapter II.9 of \cite{BH}.
\end{proof}

\subsection{Product} Let $X_{1}$ and $X_{2}$ be two metric spaces. The product $X=X_{1}\times X_{2}$ has a natural metric which is called the product metric. Let $\gamma_{i}$ be an isometry of $X_{i}$ $(i=1,2)$. It is obvious that $\gamma=(\gamma_{1},\gamma_{2})$ is an isometry of $X$ under the natural action. The following lemma tells when the converse is true.
\begin{lemma}\label{product}
Let $X=X_{1}\times X_{2}$. Then an isometry $\gamma$ on $X$ decomposes as $(\gamma_{1},\gamma_{2})$, with $\gamma_{i}$ be an isometry of $X_{i}$ $(i=1,2)$, if and only if, for every $x_{1}\in X_{1}$, there exists a point denoted $\gamma_{1}\cdot x_{1} \in X_{1}$ such that $\gamma \cdot (\{x_{1}\}\times X_{2})=\{\gamma_{1}\cdot x_{1}\}\times X_{2}$.  
\end{lemma}
\begin{proof}
See proposition 5.3 on page 56 in chapter I.5 of \cite{BH}.
\end{proof}

The following product decomposition theorem will be applied for several times in this article.
\begin{proposition}\label{pdt}
Assume that $M$ is a complete CAT(0) space. Let $c:\mathbb{R} \to M$ be a geodesic line and $P_{c}$  be the set of geodesic lines which are parallel to $c$. Then, $P_{c}$ is isometric to the product $P'_{c}\times \mathbb{R}$ where $P'_{c}$ is a closed convex subset in $M$.
\end{proposition}
\begin{proof}
See theorem 2.14 on page 183 in chapter II.2 of \cite{BH}.
\end{proof}

\subsection{Semi-uniformly visible CAT(0) spaces}
 
Recall a metric space $M$ is called \textsl{Gromov-hyperbolic} if there exists a $\delta>0$ such that every geodesic triangle is $\delta$-thin. Where a \textsl{$\delta$-thin geodesic triangle} means that each of its sides is contained in the $\delta$-neighborhood of the union of the other two sides. A $\mathbb{R}$-tree is a Gromov-hyperbolic space which holds for any $\delta>0$. For more details one can see \cite{BH, Gromov87}. The following proposition tells that Gromov-hyperbolic CAT(0) spaces are stronger than visible CAT(0) spaces.
\begin{proposition}\label{gmhiv}
Every complete Gromov-hyperbolic CAT(0) space is a visible CAT(0) space.
\end{proposition}

\begin{proof}
See proposition 10.1 in \cite{BU}. For the proper case, one can also see proposition 1.4 in chapter III.H of \cite{BH}. 
\end{proof}

Before we define semi-uniformly visible CAT(0) spaces, let us consider the following two examples.

\begin{example}\label{exmp1}
For every positive integer $n$, let $I_{n}$ be the vertical segment in $\mathbb{R}^2$ as follows
\[I_{n}:=\{(x,y)\in \mathbb{R}^2; \ \ x=\frac{1}{n}, 0\leq y \leq n\}\]
and $I$ be the horizontal segment with unit length as follows
\[I:=\{(x,y)\in \mathbb{R}^2; \ \ 0\leq x \leq 1, y\equiv0\}.\]
Consider the space 
\[M:=((\bigcup_{n=1}^{+\infty} I_{n})\bigcup I,ds^2)\]
where $ds^2$ is the induced metric from $\mathbb{R}^2$.

It is easy to see that $M$ is a complete tree. So $M$ is a complete Gromov-hyperbolic space. By proposition \ref{gmhiv}, $M$ is also a complete CAT(0) visible space. It is not hard to see that $M$ is not locally compact (around the point $(x,y)=(0,0))$, unbounded, and the visual boundary of $M$ is empty.    

\end{example}

Let $M$ be a complete CAT(0) space. We call $M$ has \textsl{an infinite-flat-strip} if there exists a totally geodesical convex subset $U\times \mathbb{R} \subseteq M$ where $U$ is unbounded. From proposition \ref{pdt} we can always assume that $U$ is closed convex. The following example tells that the convex hull of an infinite-flat-strip may not contain a half flat plane.

\begin{example}\label{exmp2}
Let $M$ be the metric space in example \ref{exmp1}. Consider the product space

\[N:=M\times \mathbb{R}\]
which is endowed with the product metric.

It is obvious that $N$ is a complete CAT(0) space, but not locally compact. Actually the visual boundary $N(\infty)$ of $N$ consists of two points, which are the positive and negative directions of the $\mathbb{R}$ component. In particular they can be joined by a geodesic line in $N$. Thus $N$ is a complete visible CAT(0) space. It is easy to see that $N$ contains an infinite-flat-strip. And $N$ is not a Gromov-hyperbolic CAT(0) space.

\end{example}
Example \ref{exmp2} tells us that a complete visible CAT(0) space may contain an infinite-flat-strip. The following definition will exclude complete CAT(0) spaces with infinite-flat-strip.

\begin{definition}\label{dosuvs}
Let $M$ be a complete CAT(0) space. $M$ is called \textsl{semi-uniformly visible} if $M$ is visible and $M$ can not contain an infinite-flat-strip. 
\end{definition}

It is obvious that any complete simply connected Riemannian manifold whose sectional curvatures are bounded by a negative number is a semi-uniformly visible CAT(0) space. The following proposition tells that semi-uniformly visible CAT(0) spaces contain a lot of standard visible spaces.
\begin{proposition}\label{sosuv}
(1): Every complete proper visible CAT(0) space is a semi-uniformly visible CAT(0) space.\\ 
(2): Every complete Gromov-hyperbolic CAT(0) space is a semi-uniformly visible CAT(0) space.
\end{proposition}

\begin{proof}
\textit{Proof of part (1):} If not. Then we can assume that $M$ is a complete proper visible CAT(0) space which contains an infinite-flat-strip $U\times \mathbb{R}$ where $U$ is unbounded. Let $c:(-\infty,+\infty)\to M$ be the geodesic line $x_{0}\times \mathbb{R}$ where $x_{0}\in U$. And let $P_{c}$ be the set of geodesic lines which are parallel to $c$. By proposition \ref{pdt}, $P_{c}$ is isometric to the product $P'_{c}\times \mathbb{R}$ where $P'_{c}$ is a closed convex subset in $M$. Hence $U\subseteq P'_{c}$. Since $U$ is unbounded, $P'_{c}$ is also unbounded. Thus $P'_{c}$ a complete, unbounded, proper space. 
The Arzel\`a Ascoli theorem would guarantee that there exists a geodesic ray $d:[0,+\infty)\to P'_{c}$. Hence $M$ contains a flat half plane $[0,+\infty)\times \mathbb{R}$ in $M$ which is impossible because $M$ is visible.
\newline
\\
\textit{Proof of part (2):} By proposition \ref{gmhiv} it suffices to show that every complete Gromov-hyperbolic CAT(0) space is semi-uniform. Assume not. Then we can assume that $M$ is a complete Gromov-hyperbolic CAT(0) space which contains an infinite-flat-strip $U\times \mathbb{R}$ where $U$ is unbounded. Let $\delta>0$ be the number such that every geodesic triangle in $M$ is $\delta$-thin. Let $c:(-\infty,+\infty)\to M$ be the geodesic line $x_{0}\times \mathbb{R}$ where $x_{0}\in U$. And let $P_{c}$ be the set of geodesic lines which are parallel to $c$. By proposition \ref{pdt}, $P_{c}$ is isometric to the product $P'_{c}\times \mathbb{R}$ where $P'_{c}$ is a closed convex subset in $M$. Hence $U\subseteq P'_{c}$.  Since $U$ is unbounded, we can find a flat strip $[0,k]\times \mathbb{R}$ with width $k$ where $k$ is an arbitrary positive number. If we choose $k$ to be greater enough than $\delta$, it is not hard to find a geodesic triangle $\Delta$ in $[0,k]\times \mathbb{R}$ such that $\Delta$ is not $\delta$-thin, which is a contradiction.   
    
\end{proof}

\begin{remark}\label{suv-1}
From part (1) of the proposition above, if the space $M$ is complete, locally compact, CAT(0) space, then semi-uniform visibility is equvilent to visibility. Hence, the interesting aspect for semi-uniform visibility is for complete CAT(0) spaces which are not locally compact. 
\end{remark}

\subsection{Horospheres in CAT(0) spaces} Let $M$ be a complete CAT(0) space and $x \in M(\infty)$. $c:[0,+\infty)\to M$ is a geodesic ray with $c(+\infty)=x$, which determines a \textsl{Busemann function} $f=f_{c}$ at $x$ given by $f(p)=\lim_{t\to +\infty}(dist(p, c(t))-t)$. $f$ is a convex function. For each $t\in \mathbb{R}$, the set $B(t):=\{q \in M, f(q)\leq t\}$ is a \textsl{horoball} at $x$, whose boundary,  a level set of $f$, is a \textsl{horosphere} at $x$. Every horoball is convex because of the convexity of $f$.  If $x$ is fixed by some parabolic isometry $\gamma$ of $M$, then each horoball at $x$ is $\gamma$-invariant. In particular each horosphere at $x$, the boundary of some horoball, is also $\gamma$-invariant. One can see more details in \cite{BH}. The following lemma will be applied to prove theorem \ref{mt-2}.

\begin{proposition}\label{hsie}
Let $M$ be a complete, simply connected, nonpositively curved Riemannian manifold and $c:[0,+\infty)\to M$ be a geodesic ray. Then every horosphere at $c(+\infty)$ is diffeomorphic to $\mathbb{R}^{k}$ where $k=dim(M)-1$.
\end{proposition}   

\begin{proof}
Let $H$ be a horosphere at $c(+\infty)$ and $B$ be the horoball whose boundary is $H$. Choose a point $q$ in the complement of $B$ and consider the distance function $f:H \to (0,+\infty)$, given by $f(p)=dist(q,p)$. The idea is prove that $f$ only has one critical point on $H$, which is the intersection point of $H$ and the geodesic ray, starting at $q$, which goes to $c(+\infty)$. Then the conclusion follows from standard Morse theory. One can see lemma 3-4 on page 511 in \cite{Fukaya84} for details.
\end{proof}

\subsection{Isometries on CAT(0) spaces}
Let $M$ be a complete CAT(0) space and $\gamma$ be an isometry of $M$. If the translation length $|\gamma|$ is achieved in $M$, $\gamma$ is either elliptic or hyperbolic. The following lemma tells us any isometry on $\mathbb{R}^n$ can achieve its translation length, which will be applied in subsequent sections.
\begin{lemma}\label{R1}
Let $\gamma$ be an isometry of the Euclidean space $\mathbb{R}^n$ where $n$ is a positive integer. Then $\gamma$ is either elliptic or hyperbolic.
\end{lemma}
\begin{proof}
One possible way is to argue it by contradiction. Assume that $\gamma$ is parabolic. Then there exists a horosphere $H$ such that $\gamma$ acts on $H$ as an isometry. Since $H$ is a horosphere of $\mathbb{R}^n$, $H$ is isometric to $\mathbb{R}^{n-1}$ which is totally geodesic in $\mathbb{R}^n$. The conclusion follows by induction on the dimensions. For the details we leave it as an exercise to the reader.
\end{proof}

From definition an elliptic isometry $\gamma$ of $M$ has at least one fixed point. The following is one of the basic fixed point theorems in CAT(0) geometry.
\begin{lemma}\label{fpt}
An isometry $\gamma$ of a complete CAT(0) space $M$ has a fixed point if and only if there exists a $\gamma$-invariant bounded subset in $M$.
\end{lemma}
\begin{proof}
The idea is to prove that, for any bounded subset in $M$ there exists a unique centre point for this bounded subset, and then prove this point is fixed by $\gamma$. One can see the proof of proposition 6.7 on page 231 in chapter II.6 of \cite{BH}. 
\end{proof}

Recall that the translation length $\gamma$ of an isometry of $M$ is defined as the infimum of translation of $\gamma$ over $M$. The following lemma gives us another viewpoint for the translation length of a single isometry.

\begin{lemma}\label{tl}
Let $M$ be complete CAT(0) space and $\gamma$ be an isometry of $M$. Then, for all $p\in M$, we have 
\[|\gamma|=\lim_{n\rightarrow+\infty}{\frac{dist(\gamma^{n} \cdot p,p)}{n}}.\]
\end{lemma}
\begin{proof}
See lemma 6.6 on page 83 in \cite{BGS} or lemma 3.4 in \cite{Wu12}.
\end{proof}

As a direct corollary,
\begin{lemma}\label{22}
Let $M$ be complete CAT(0) space and $\gamma$ be an isometry of $M$. Then
\[|\gamma|^2=2|\gamma|.\]
\end{lemma}
\begin{proof}
Let $p \in M$. By lemma \ref{tl}, 
\[|\gamma^2|=\lim_{n\rightarrow+\infty}{\frac{dist(\gamma^{2n} \cdot p,p)}{n}}=2 \lim_{n\rightarrow+\infty}{\frac{dist(\gamma^{2n} \cdot p,p)}{2n}}=2|\gamma|.\]
\end{proof}

A group $G$ acting on a metric space $X$ is called \textsl{properly} if for each compact subset $K\subset X$, the set $K \cap gK$ is nonempty for only finitely many $g$ in $G$. 
\begin{lemma}\label{z2ip}
Let $M$ be a complete visible CAT(0) Riemannian manifold and $G=\mathbb{Z}\oplus \mathbb{Z}$, a free abelian group of rank 2, act properly on $M$ by isometries. Then for any nontrivial $g\in G$, $g$ is parabolic. 
\end{lemma}

\begin{proof}
Since $G$ acts properly on $M$, for any nontrivial $g\in G$, $g$ is either parabolic or hyperbolic. Assume that $g$ is hyperbolic. Let $Min(g):=\{p\in M, dist(g\cdot p,p)=|g|\}$. From theorem 6.8 on page 231 in chapter II.6 of \cite{BH} we know that $Min(g)$ is isometric to the product $Y\times \mathbb{R}$ on which $g$ acts trivially on the $Y$ component, where $Y$ is a closed convex subset in $M$. Since $M$ is a visible CAT(0) manifold, $Y$ is bounded, otherwise there exists a flat half plane $[0,+\infty)\times \mathbb{R}$ in $X$ which is impossible in visible CAT(0) spaces (this step needs that the space is proper). Let $h \in G$ such that the group $<g,h>$, generated by $g$ and $h$, is a free abelian group of rank 2. Since $gh=hg$, $Min(g)$ is $h$-invariant. Since $h$ is an isometry, it must send a geodesic line to another geodesic line. Thus, by lemma \ref{product}, $h$ splits as $(h_{1},h_{2})$ where $h_{1}$ is an isometry on $Y$ and $h_{2}$ is an isometry on $\mathbb{R}$. Firstly $Y$ is a complete CAT(0) space because $Y$ is closed convex in $M$. Since $Y$ is bounded, by lemma \ref{fpt}, there exists $x_{1} \in Y$ such that $h_{1} \cdot x_{1}=x_{1}$. Hence $x_{1}\times \mathbb{R}$ is $<g,h>$-invariant. Since $G$ acts properly on $M$ and $<g,h>$ is a subgroup of $G$, $<g,h>$ also acts properly on $x_{1}\times \mathbb{R}$, which is impossible because the rank of $<g,h>$ is 2.
\end{proof}

\subsection{Mapping class groups} Let $S_{g,n}$ be a Riemann surface of genus $g$ with $n$ punctures, and $\Mod_{S_{g,n}}$ be the mapping class group of $S_{g,n}$, i.e., the group of isotopy classes of self-homeomorphisms of $S_{g,n}$ which preserve the orientation and the punctures. The following proposition lists the basic properties of $\Mod_{S_{g,n}}$, which will be used.
\begin{proposition}\label{mcgp} 
Let $\Mod_{S_{g,n}}$ be the mapping class group of $S_{g,n}$. Then

(1): $\Mod_{S_{g,n}}$ acts properly  on the Teichm\"uller space.

(2): $\Mod_{S_{g,n}}$ is finitely generated by Dehn-twists along simple closed curves.

(3): There exist torsion-free subgroups of finite index in $\Mod_{S_{g,n}}$. 
\end{proposition}
\begin{proof}
See the details in \cite{FM}.
\end{proof}

Bestvina, Kapovich and Kleiner in \cite{BKK} defined the \textsl{action dimension} of a group $G$, denoted by $actdim(G)$, to be the minimum dimension of a contractible manifold on which $G$ properly acts. For examples, the action dimension of $\mathbb{Z}$ is $1$. The action dimension of the fundamental group of a closed hyperbolic surface is $2$.  In \cite{BKK} there is a notation, which is called the \textsl{obstructor dimension} of a group $G$, denoted by $obdim(G)$ (see \cite{BKK}).
\begin{proposition}[BKK] \label{obdim}
(1): For any group $G$, $actdim(G)\geq obdim(G)$.

(2): The obstructor dimension is a quasi-isometric invariance.

\end{proposition}
\begin{proof}
See theorem 1 in \cite{BKK} for part (1) and remark 11 in \cite{BKK} for part (2).
\end{proof}

In \cite{Despto} Z. Despotovic proved that the action dimension of the mapping class group $actdim(\Mod_{S_{g,n}})=6g-6+2n$ if $3g+n\geq 5$ (see page 6 of \cite{FW10}). In fact, she proved the obstructor dimension $obdim(\Mod_{S_{g,n}})=6g-6+2n$ and then used part (1) of proposition \ref{obdim} to conclude $actdim(\Mod_{S_{g,n}})=6g-6+2n$. The following result tells that the action dimension is preserved by subgroups of mapping class groups, up to finite index.
\begin{proposition}[Despotovic] \label{adom}
If $3g+n\geq 5$, then for any finite index subgroup $G$ of $\Mod_{S_{g,n}}$, $actdim(G)=6g-6+2n$.
\end{proposition}

\begin{proof}
Let $G$ be a subgroup of $\Mod_{S_{g,n}}$ with finite index. So $G$ endowed with a word metric is quasi-isometric to $\Mod_{S_{g,n}}$. Since the obstructor dimension of $\Mod_{S_{g,n}}$ $obdim(\Mod_{S_{g,n}})=6g-6+2n$ (see \cite{Despto}), by part (2) of proposition \ref{obdim} we have $obdim(G)=6g-6+2n$. Thus from part (1) of proposition \ref{obdim} we know that $actdim(G)\geq 6g-6+2n$.

On the other hand it is obvious that $actdim(G)\leq 6g-6+2n$ because $G$ acts properly on the Teichm\"uller space which is contractible. Hence $actdim(G)=6g-6+2n$. 

\end{proof} 

\section{Proofs of theorem \ref{mt-1} and theorem \ref{class}}\label{mt12}
Before we go to prove theorem \ref{mt-1}, let us control the size of the fixed points of parabolic isometries.

\begin{proposition}\label{lt1}
Let $M$ be a complete semi-uniformly visible CAT(0) space and $\gamma$ be a parabolic isometry on $M$. Then, the number of the fixed points of $\gamma$ satisfy
\[\#\{x \in M(\infty): \gamma \cdot x=x\}\leq 1.\]
\end{proposition}

\begin{proof}
Assume not. That is $|Fix(\gamma)|\geq 2$, where $Fix(\gamma)$ denotes the set of the fixed points of $\gamma$ in $M(\infty)$. Let $x\neq y \in Fix(\gamma)$. Since $M$ is visible, there exists a geodesic line $c: \mathbb{R}\to M$ such that $c(-\infty)=x$ and $c(+\infty)=y$. Let $P_{c}$ be the set of geodesic lines which are parallel to $c$. By proposition \ref{pdt}, $P_{c}$ is isometric to the product $P'_{c}\times \mathbb{R}$ where $P'_{c}$ is a convex subset in $M$. Since $M$ is semi-uniformly visible, $P'_{c}$ is bounded, otherwise there exists an infinite-flat-strips in $M$, which contradicts the definition of semi-uniformly CAT(0) spaces.

Since $c(-\infty),c(+\infty) \in Fix(\gamma)$, $\gamma \cdot (c(\mathbb{R}))$ is also geodesic line which is parallel to $c$. In particular, $P_{c}=P'_{c}\times \mathbb{R}$ is a $\gamma$-invariant subset in $M$. From lemma \ref{product}, $\gamma$ splits as $(\gamma_{1},\gamma_{2})$ where $\gamma_{1}$ is an isometry on $P'_{c}$ and $\gamma_{2}$ is an isometry on $\mathbb{R}$. Firstly since $P'_{c}$ is convex in $M$, $P'_{c}$ is also a CAT(0) space. Since $P'_{c}$ is bounded, by lemma \ref{fpt}, there exists $x_{1} \in P'_{c}$ such that $\gamma_{1} \cdot x_{1}=x_{1}$. By lemma \ref{R1}, $\gamma_{2}$ is either elliptic or hyperbolic. 

\textsl{Case 1: $\gamma_{2}$ is elliptic.}

There exists $x_{2}\in \mathbb{R}$ such that $\gamma_{2}\cdot x_{2}=x_{2}$. In particular $\gamma=(\gamma_{1},\gamma_{2})$ fixes the point $(x_{1},x_{2})$, which means that $\gamma$ is elliptic, which contradicts the assumption that $\gamma$ is parabolic.

\textsl{Case 2: $\gamma_{2}$ is hyperbolic.}

Since $\gamma_{1}\cdot x_{1}=x_{1}$, $\gamma=(\gamma_{1},\gamma_{2})$ acts on the line $x_{1}\times \mathbb{R}$ as a translation. From the definition of hyperbolic isometries, $x_{1}\times \mathbb{R}$ is an axis for $\gamma$. In particular, $\gamma$ is hyperbolic, which also contradicts our assumption that $\gamma$ is parabolic.

\end{proof}

\begin{remark}\label{2}
If $M$ is proper, i.e., locally compact,  it is not hard to prove the existence of fixed points for parabolic isometry (see proposition 8.25 on page 275 in chapter II.8 of \cite{BH}). By part (1) of proposition \ref{sosuv}, proposition \ref{lt1} tells that that the fixed point of any parabolic isometry of a complete proper visible CAT(0) space is unique. One can also see \cite{FNS}.
\end{remark}

\begin{remark}\label{efgh}
If $M$ is a complete Gromov-hyperbolic CAT(0) space (may be no proper!), the existence of fixed points of any parabolic isometry is also guaranteed (see theorem 0.3 in \cite{BU} or one can also see \cite{Ghys}). I am grateful to Tushar Das for this point. 
\end{remark}  


\begin{proposition}\label{gt2}
Let $M$ be a complete CAT(0) space and $\gamma$ be an isometry on $M$. If $|\gamma|>0$, then 
\[\#\{x \in M(\infty): \gamma \cdot x=x\}\geq 2.\]
\end{proposition}
\begin{proof}
Since $|\gamma|>0$, $\gamma$ is either hyperbolic or parabolic.

If $\gamma$ is hyperbolic, let $c:(-\infty,+\infty) \to M$ be an axis for $\gamma$. The conclusion follows from the fact that $\{c(-\infty),c(+\infty)\}$ belongs to $\{x \in M(\infty): \gamma \cdot x=x\}$.

If $\gamma$ is parabolic. Since $|\gamma|>0$, a special case of a result of Karlsson and Margulis \cite{KM} shows that $\gamma$ has a unique fixed point $x\in M(\infty)$ such that for every $p \in M$ and every geodesic ray $c:[0,+\infty)\to M$ with $c(0)=p$ and $c(+\infty)=x$ we have 
\[\lim_{n\rightarrow +\infty}{\frac{dist(\gamma^n \cdot p, c(n |\gamma|))}{n}}=0.\]

Similarly $|\gamma^{-1}|=|\gamma|>0$, $\gamma^{-1}$ has a unique fixed point $y\in M(\infty)$ such that for every $p \in M$ and every geodesic ray $c':[0,+\infty)\to M$ with $c'(0)=p$ and $c'(+\infty)=y$ we have 
\[\lim_{n\rightarrow +\infty}{\frac{dist(\gamma^{-n} \cdot p, c'(n |\gamma|))}{n}}=0.\]

By triangle inequality, we have 
\[\lim_{n\rightarrow+\infty}{\frac{dist(c(n |\gamma|), c'(n |\gamma|))}{n}}=\lim_{n\rightarrow+\infty}\frac{dist(\gamma^{n} \cdot p,\gamma^{-n}\cdot p)}{n}.\]

Since $\gamma$ is an isometry, by lemma \ref{pro:angular} and lemma \ref{tl}, we have
\[2 |\gamma| \sin(\frac{\angle(c(+\infty),c'(+\infty))}{2})=|\gamma^2|.\]

From lemma \ref{22}, we know that 
\[2 |\gamma| \sin(\frac{\angle(c(+\infty),c'(+\infty))}{2})=2|\gamma|.\]

Since $|\gamma|\neq 0$, $\angle(c(+\infty),c'(+\infty))=\pi \neq 0$. That is $\angle(x,y)\neq 0$. Since $x, y \in Fix(\gamma)$, we have
\[\#\{x \in M(\infty): \gamma \cdot x=x\}\geq 2.\]
 
\end{proof}

\begin{remark}
The result of Karlsson and Margulis above was also used by M. Bridson in \cite{Bridson10} to show that any Dehn-twist has zero translation length when a mapping class group of a closed surface with genus $g\geq 3$ acts on a complete CAT(0) space by isometries. 
\end{remark}

Now we are ready to prove theorem \ref{mt-1} and theorem \ref{class}.

\begin{proof}[Proof of theorem \ref{mt-1}] 
Suppose not, we assume that $|\gamma|>0$. From proposition \ref{gt2} we know that 
\[\#\{x \in M(\infty): \gamma \cdot x=x\}\geq 2.\]
On the other hand, since $M$ is complete semi-uniformly visible and $\gamma$ is parabolic, by proposition \ref{lt1}, we have
\[\#\{x \in M(\infty): \gamma \cdot x=x\}\leq 1\]
which is a contradiction.

\end{proof}

Since a complete proper visible CAT(0) space is semi-uniformly visible (see part (1) of proposition \ref{sosuv}), theorem \ref{mt-1} implies
\begin{theorem}(\cite{Wu12})
Let $M$ be a complete proper visible CAT(0) space. Then any parabolic isometry of $M$ has zero translation length. i.e., for any parabolic isometry $\gamma$ of $M$ we have $|\gamma|=0$.
\end{theorem}

Since a complete Gromov-hyperbolic CAT(0) space is semi-uniformly visible (see part (2) of proposition \ref{sosuv}), by theorem \ref{mt-1} we immediately obtain
\begin{theorem}[Buyalo]
Let $M$ be a complete Gromov-hyperbolic CAT(0) space. Then any parabolic isometry of $M$ has zero translation length. i.e., for any parabolic isometry $\gamma$ of $M$ we have $|\gamma|=0$.
\end{theorem}

\begin{remark}
We call a manifold $M$ is tame if $M$ is the interior of some compact manifold $\overline{M}$ with boundary. Phan conjectures in \cite{Phan} that let $M$ be a tame, finite volume, negatively curved manifold, then $M$ is not visible if the fundamental group of $M$ contains a parabolic isometry of $\tilde{M}$ with positive translation length. This conjecture is confirmed in \cite{Wu12}. Since the proof in this paper is different as the one in \cite{Wu12}, theorem \ref{mt-1} gives a new proof to confirm this conjecture. 
\end{remark}

\begin{remark}\label{tushar}
 A complete visible CAT(0) space may contain parabolic isometries with positive translation lengths. One possible example is the following: \textsl{Firstly find a complete CAT(0) space $M$ such that there exists a parabolic isometry $\gamma_{0}$ on it and the visual boundary is empty (one may construct $M$ by taking the convex hull of the orbits of zero under $\alpha$ on page 6 in section 1.3.3 in \cite{Vale}, which is modified by Edelstein's example in \cite{Edel64}). Then take the product of $M$ with $\mathbb{R}$ we get a space $N:=M\times \mathbb{R}$. Since the visual boundary of $N$ consists of two points which can be joined by a geodesic line, $N$ is a complete visible CAT(0) space. Consider the isometry $\gamma:N \to N$ which is defined as $\gamma \cdot (m,t)=(\gamma_{0}\cdot m,t+1)$. It is easy to see that $\gamma$ is a parabolic isometry on $N$ whose translation length $|\gamma|>0$.} I am greatly indebted to Tushar Das to share his idea on this example.  

\end{remark}

Theorem \ref{class} is an easy application of theorem \ref{mt-1}.

\begin{proof}[Proof of theorem \ref{class}]  \texttt{Proof of (1)}: By definition.

\texttt{Proof of (2)}: If $\gamma$ is hyperbolic, by the definition we know that $|\gamma|>0$.

If $|\gamma|>0$, assume that $\gamma$ was not hyperbolic, so $\gamma$ is parabolic. From theorem \ref{mt-1} we have $|\gamma|=0$ which contradicts with our assumption. 

\texttt{Proof of (3)}: If $\gamma$ is parabolic, it is obvious that $\gamma$ does not have fixed points. $|\gamma|=0$ follows from theorem \ref{mt-1}. 

If $\gamma$ does not have fixed points and $|\gamma|=0$, the conclusion that $\gamma$ is parabolic follows from part (2).
\end{proof}

\section{Proof of theorem \ref{mt-2}}\label{somp}
Before we go to prove theorem \ref{mt-2}, let us make some preparations. The following proposition has been proven in different literature (see \cite{BF, KN04, McP99}). For completeness we still give the proof here. 

\begin{proposition}\label{fpfm}
Let $M$ be a complete, visible CAT(0) Riemannian manifold. Assume that the mapping class group $\Mod_{S_{g,n}}$ acts properly on $M$. If $3g+n\geq 5$, then there exists a point $x\in M(\infty)$ such that $\gamma \cdot x=x$ for all $\gamma \in \Mod_{S_{g,n}}$.
\end{proposition}

\begin{proof}

Let $\alpha$ be a simple closed curve on $S_{g,n}$ and $\tau_{\alpha}$ be the Dehn twist along $\alpha$ (see the definition of Dehn twist in \cite{FM}). Since $3g+n\geq 5$, there exists a simple closed curve $\beta$ which is disjoint with $\alpha$. Let $\tau_{\beta}$ be the Dehn twist along $\beta$. Since $<\tau_{\alpha},\tau_{\beta}>$ is a free abelian group of rank 2, by lemma \ref{z2ip}, $\tau_{\alpha}$ is parabolic. Since $M$ is proper visible, by remark \ref{2} there exists a unique $x \in M(\infty)$ such that $Fix(\tau_{\alpha})=\{x\}$. 

\textsl{Claim 1: If $\tau_{\alpha}\cdot \tau_{\beta}=\tau_{\beta}\cdot \tau_{\alpha}$, then $Fix(\tau_{\beta})=Fix(\tau_{\alpha})=\{x\}$.}

\textsl{Proof claim 1:} Since $\tau_{\alpha}\cdot \tau_{\beta}=\tau_{\beta}\cdot \tau_{\alpha}$ and $\tau_{\alpha}\cdot x=x$, $\tau_{\alpha}\cdot (\tau_{\beta}\cdot x)=\tau_{\beta}\cdot x$. $Fix(\tau_{\beta})=\{x\}$ follows from the fact that $Fix(\tau_{\alpha})=\{x\}$.

\textsl{Claim 2: For any simple closed curve $\beta$ on $S_{g,n}$, $Fix(\tau_{\beta})=\{x\}$.}

\textsl{Proof claim 2:} Let $\beta$ be a simple closed curve. Since $3g+n\geq 5$, the curve complex is connected (see theorem 4.3 in \cite{FM}). In particular, there exists a sequence of simple closed curves $\{\alpha_{i}\}_{i=1}^{k}$ such that $\alpha_{1}=\alpha, \alpha_{k}=\beta$ and $\alpha_{i}\bigcap \alpha_{i+1}=\emptyset$. Hence, by claim 1, 
\[Fix(\tau_{\beta})=Fix(\tau_{\alpha_{k-1}})=Fix(\tau_{\alpha_{k-2}})=\cdots =Fix(\tau_{\alpha_{1}})=\{x\}.\]

The conclusion follows from part (2) of proposition \ref{mcgp} and claim 2.
\end{proof}

\begin{proposition}\label{mcape}
Let $M$ be a complete, visible CAT(0) Riemannian manifold and the mapping class group $\Mod_{S_{g,n}}$ acts properly on $M$. If $3g+n\geq 5$, then any infinite ordered element $\phi \in \Mod_{S_{g,n}}$ acts as a parabolic isometry.
\end{proposition}
\begin{proof}
Suppose that there exists an element $\phi \in \Mod_{S_{g,n}}$ with infinite order which acts on $M$ as a hyperbolic isometry. Then there exists a geodesic line $\gamma:\mathbb{R}\rightarrow M$, an axis for $\phi$, such that $\phi \cdot \gamma(t)=\gamma(|\phi|+t)$ for all $t \in \mathbb{R}$. Since $M$ is a visible CAT(0) space, it is not hard to see that $Fix(\phi)=\{\gamma(+\infty),\gamma(-\infty)\}$. From proposition \ref{fpfm} we assume that $\gamma(+\infty)$ is fixed by $\Mod_{S_{g,n}}$.  Let $\sigma \in \Mod_{S_{g,n}}$, since $\sigma$ fixes $\gamma(+\infty)$ there exists some a $C>0$ such that $dist(\sigma \cdot \gamma(n\cdot|\phi|),\gamma(n\cdot|\phi|))\leq C$ for all $n>0$. Hence $dist((\phi^{-n}\cdot \sigma\cdot \phi^n)\cdot \gamma(0),\gamma(0))\leq C$. Since the action is proper, there exists a subsequence $\{n_{i}\}$ such that $\phi^{-n_{i}}\cdot \sigma \cdot \phi^{n_{i}}\equiv \phi^{-n_{1}}\cdot \sigma\cdot \phi^{n_{1}}$, hence $\phi^{n_{1}-n_{i}}\cdot \sigma= \sigma\cdot \phi^{n_{1}-n_{i}}$. Since $\sigma$ is arbitrary and $\phi$ has infinite order in $\Mod_{S_{g,n}}$, we can choose $\sigma$ to be pseudo-Anosov such that $\{\sigma,\phi\}$ generates a free group of rank 2 (see \cite{Iv92}). Since $\phi^{n_{1}-n_{i}}\cdot \sigma= \sigma\cdot \phi^{n_{1}-n_{i}}$, the group $<\sigma,\phi>$ contains a free abelian subgroup of rank $2$, which is a contradiction since $<\sigma,\phi>$ is a free group.     
\end{proof}

\begin{proposition}\label{mcghs}
Let $M$ be a complete, visible CAT(0) Riemannian manifold and the mapping class group $\Mod_{S_{g,n}}$ acts properly on $M$. If $3g+n\geq 5$, then there exists a horosphere $H$ such that every torsion free subgroup of $\Mod_{S_{g,n}}$ acts properly on $H$. 
\end{proposition}

\begin{proof}
From proposition \ref{fpfm}, there exists a point $x \in M(\infty)$ such that $\Mod_{S_{g,n}}$ fixes $x$. Let $H$ be a horosphere at $x$ and $G$ be a torsion free subgroup of $\Mod_{S_{g,n}}$. By proposition \ref{mcape} we know that $G$ consists of parabolic isometries except the unit. Hence, by proposition 8.25 on page 275 in chapter II.8 of \cite{BH}, $H$ is $G$-invariant because $G$ fixes $x$. Let $d$ be the metric of $M$ and  $d_{H}$ be the induced metric on $H$. It is obvious that $d_{H}(p,q)\geq d(p,q)$ for all $p,q \in H$. The conclusion that $G$ acts properly on $H$ follows easily from the facts that $G$ acts properly on $M$ and $d_{H}(p,q)\geq d(p,q)$.  
\end{proof}

Now we are ready to prove theorem \ref{mt-2}.

\begin{proof}[Proof of theorem \ref{mt-2}]
We argue it by contradiction. Assume that $\mathbb{M}(S_{g,n})$ admits a complete, visible CAT(0) Riemannian metric $ds^2$. Let $\mathbb{T}(S_{g,n})$ be the universal covering space of $(\mathbb{M}(S_{g,n}),ds^2)$, which is the Teichm\"uller space endowed with the pull back metric $ds^2$. By part (3) of proposition \ref{mcgp} we can assume that $G$ is a torsion free subgroup of $\Mod_{S_{g,n}}$ with finite index. By proposition \ref{mcghs}, there exists a horosphere $H$ such that $G$ acts properly on $H$. Proposition \ref{hsie} tells us that $H$ is homeomorphic to $\mathbb{R}^{6g-7+2n}$, in particular $H$ is a contractible manifold. By the definition of action dimension we know that 
\[actdim(G)\leq 6g-7+2n.\] 
On the other hand, since $3g+n\geq 5$ and $G$ is a finite index subgroup of $\Mod_{S_{g,n}}$, by proposition \ref{adom}, we have
\[actdim(G)=6g-6+2n\]
which is a contradiction. 
\end{proof}

\begin{remark} If we carefully check the proof of theorem \ref{mt-2}, we can conclude as follows: \textsl{The Teichm\"uller space $\mathbb{T}(S_{g,n})$ $(3g+n\geq 5)$ admits no complete $\Mod_{S_{g,n}}$-invariant CAT(0) Riemannian metric such that every Dehn twist has only one fixed point in the visual boundary of $\mathbb{T}(S_{g,n})$.} I am very grateful to Benson Farb for this point.
\end{remark}

\begin{remark}
If we consider the case that $n=0$ and $g\geq 3$. From lemma 4.5 in \cite{Wu12}, the proof of theorem \ref{mt-2} can also conclude as follows: \textsl{The Teichm\"uller space $\mathbb{T}(S_{g})$ $(g\geq 3)$ admits no complete $\Mod_{S_{g}}$-invariant CAT(0)  Riemannian metric such that the Dehn twist on some non-separating simple closed curve has only one fixed point in the visual boundary of $\mathbb{T}(S_{g})$.}
\end{remark}
  
\bibliographystyle{amsalpha}
\bibliography{ref}  

\end{document}